\documentclass[10pt]{amsart}
\usepackage[letterpaper, margin=0.8in]{geometry}
\usepackage{graphicx}
\usepackage{amsmath,amssymb,amsthm,mathtools}
\usepackage[all]{xy}
\usepackage[utf8]{inputenc}
\usepackage[hidelinks]{hyperref}
\usepackage{enumitem}
\usepackage{xcolor}
\newtheorem{theorem}{Theorem}

\newtheorem{lemma}[theorem]{Lemma}
\newtheorem{proposition}[theorem]{Proposition}
\newtheorem{conjecture}[theorem]{Conjecture}

\usepackage{pifont}

\newtheorem*{theorem*}{Theorem}
\newtheorem*{conjecture*}{Conjecture}
\newtheorem*{remark*}{Remark}
\newtheorem*{question*}{Question}
\begin{document}
\title{Rational points on generic marked hypersurfaces}
\date{\today}

\author{Qixiao Ma}
\address{Insitute of Mathematical Sciences, ShanghaiTech University, 201210, Shanghai, China}
\email{maqx1@shanghaitech.edu.cn}

\begin{abstract}
Over fields of characteristic zero, we show that for $n=1,d\geq4$ or $n=2,d\geq5$ or $n\geq3, d\geq 2n$, the generic $m$-marked degree-$d$ hypersurface in $\mathbb{P}^{n+1}$ admits the $m$ marked points as all the rational points. Over arbitrary fields, we show that for $n=1,d\geq4$ or $n\geq2, d\geq 2n+3$, the identiy map is the only rational self-map of the generic degree-$d$ hypersurface in $\mathbb{P}^{n+1}$.
\end{abstract}

\maketitle

Let $g\geq3,n\geq0$ be integers. It is a well-known result  that the universal curve over $\mathcal{M}_{g,n}$ has no more $k(\mathcal{M}_{g,n})$-rational points other than the obvious ones. 
Over complex numbers, this follows from works in Teichm\"uller theory due to
Hubbard \cite{MR0294719-Hubbard}, Earle and Kra \cite{MR0425183-Earle-Kra}. Over arbitrary fields, algebraic methods were developed by Hain \cite{MR2784328-Hain} and Watanabe \cite{MR4029677-Watanabe}. 

The goal of this note is to extend the result from generic marked curves to generic marked hypersurfaces in projective spaces, see Theorem \ref{theorem-1} and Theorem \ref{theorem-2}. The proof of Theorem \ref{theorem-1} uses non-existence of rational curves on very general high-degree hypersurfaces due to Voisin \cite{MR1420353-Voisin}, \cite{MR1669712-Voisin-Correction}. The proof of Theorem \ref{theorem-2} uses a dimension count of the space of rational curves on very general hypersurfaces, due to Riedl and Woolf \cite{MR3723169-Riedl-Woolf}.

Arithmetically, our result can be viewed as a positive answer to the Franchetta-type question: 
Over the function fields of moduli spaces, 
are constructions on the generic objects necessarily canonical? Topologically, our result can be interpreted as rigidity of moduli spaces under the Guiding Principles proposed by Farb \cite[1.2,1.4]{farb2023rigidity}.

\section{Notations and Main Results}\label{setup-n-marked} 
\subsection{Notations} Let $k$ be a field. Let $n,d$ be positive integers. Let $P_{0}:=|\mathcal{O}_{\mathbb{P}^{n+1}}(d)|$ be the parameter space of degree-$d$ hypersurfaces in $\mathbb{P}^{n+1}_k$, and let $P_{1}\subset P_{0}\times\mathbb{P}^{n+1}$ be the tautological family. We call the projections $p_{1}\colon P_{1}\to P_{0}$ the tautological morphism and $q_{1}\colon P_{1}\to\mathbb{P}^{n+1}$ the canonical morphism to the ambient projective space.

Let $P_m:=P_1\times_{P_0}\cdots\times_{P_0}P_1$ be the fiber product of $m$ copies of $p_1\colon P_1\to P_0$. Let $p_m\colon P_m\to P_{m-1}$ be the projection onto the product of the first $m-1$ factors, and let $q_m=q_1\circ\mathrm{pr}_m\colon P_m\to\mathbb{P}^{n+1}$ be the canonical morphism to $\mathbb{P}^{n+1}$ induced by the projection onto the last factor. The product morphism $(p_m,q_m)\colon P_m\to P_{m-1}\times\mathbb{P}^{n+1}$ realizes $P_m$ as a family of hypersurface over $P_{m-1}$. 

For $i=1,\cdots,m$, let $\mathrm{pr}_i\colon P_m\to P_1$ be the projection onto the $i$-th factor, then the product morphism $(\mathrm{id}_{P_m},\mathrm{pr}_i)\colon P_m\to P_m\times_{P_0} P_1=P_{m+1}$ is a section to $p_{m+1}\colon P_{m+1}\to P_m$. We denote $(\mathrm{id}_{P_m},\mathrm{pr}_i)$ by $d_i$ and call $\{d_i\}_{i=1}^m$ the tautological sections to $p_{m+1}$.

Let us call $P_m$ the parameter space of $m$-marked degree-$d$ hypersurfaces in $\mathbb{P}^{n+1}$, and call $p_{m+1}\colon P_{m+1}\to P_m$ its tautological family. Let $K_m$ be the function field of $P_m$, and let $X_m$ be the generic fiber of $p_{m+1}\colon P_{m+1}\to P_m$. We call the $K_m$-variety $X_m\subset\mathbb{P}^{n+1}_{K_m}$ the generic $m$-marked degree-$d$ hypersurface in $\mathbb{P}^{n+1}$. The tautological sections $\{d_i\}_{i=1}^m$ define $m$ distinct $K_m$-rational points on $X_m$, denoted by $\{\delta_i\}_{i=1}^m$. 

\subsection{Rational Points}The first result of this note is the following:
\begin{theorem}\label{theorem-1} Let $k$ be a field of characteristic zero. Let $m\geq0$ be an integer. Let $X_m/K_m$ be the generic $m$-marked degree-$d$ hypersurface in $\mathbb{P}^{n+1}_k$. If $n=1, d\geq4$ or $n=2,d\geq5$ or 
$n\geq3, d\geq 2n$, then the set of $K_m$-rational points of $X_m$ consists of its $m$ tautological sections:
$$X_m(K_m)=\{\delta_1,\cdots,\delta_m\}.$$
\end{theorem}

\subsection{Rational Self-maps} For $m=1$, Theorem \ref{theorem-1} is equivalent to: Every rational self map of $X_0/K_0$ is the identity map. We are able to extend this result to arbitrary characteristics.
\begin{theorem}\label{theorem-2}Let $k$ be a field. Let $X_0/K_0$ be the generic degree $d$-hypersurface in $\mathbb{P}^{n+1}_k$. If $n\geq1, d\geq 2n+3$ or $(n,d)=(1,4)$, then the identity map is the only $K_0$-rational self-map of $X_0$:
$$\mathrm{Rat}_{K_0}(X_0,X_0)=\{\mathrm{id}_{X_0}\}.$$
\end{theorem}

\section{Proof of Theorem \ref{theorem-1}}
\subsection{}We begin with the case $m=0$. In this case, we can relax the restrictions on the dimension, the degree and the characteristic of the base field.
\begin{proposition}\label{m=0}
Let $k$ be any field. Let $X_0/K_0$ be the generic degree-$d$ hypersurface in $\mathbb{P}_k^{n+1}$. If $n\geq0$ and $d\geq2$, then the degree of every zero-cycle on $X_0$ is divisible by $d$, and therefore $X_0$ does not admit $K_0$-rational points:
$$X_0(K_0)=\emptyset.$$
\end{proposition}
\begin{proof}
Recall that with our notation in Section \ref{setup-n-marked}, the generic hypersurface $X_0$ is the generic fiber of the projection $$p_1\colon P_1\subset P_0\times\mathbb{P}^{n+1}\to P_0.$$ The projection onto the second factor $q_1\colon P_1\subset P_0\times\mathbb{P}^{n+1}\to\mathbb{P}^{n+1}$ realizes $P_1$ as a projective space bundle over $\mathbb{P}^{n+1}$. By \cite[Theorem 3.3]{MR1644323}, the Chow ring $\mathrm{CH}^*(P_1)$ is generated by $h_1:=p_1^*\mathcal{O}_{P_0}(1)$ and $h_2:=q_1^*\mathcal{O}_{\mathbb{P}^{n+1}}(1)$. For any zero-cycle $z$ on $X_0$, its closure $\overline{z}$ is a codimension-$(n+1)$ cycle on $P_1$. Therefore in $\mathrm{CH}^{n+1}(P_1)$, we have $$[\overline{z}]=\sum_{a+b=n+1} c_{a,b}\cdot h_1^a\cdot h_2^b,\textrm{  with } c_{a,b}\in\mathbb{Z}.$$ The degree of $z$ can be expressed as intersection number $$\deg(z)= [\overline{z}]\cdot h_1^{\dim P_0}=c_{0,n+1}\cdot h_1^{\dim P_0}\cdot h_2^{n+1}=d\cdot c_{0,n+1},$$ where the last equality follows from Bezout's theorem. This finishes the proof.
\end{proof}

\subsection{}In order to deal with the case $m>0$, we collect some lemmas. 
\begin{lemma}\label{Riedl-Woolf-lemma}
For $d\geq 2n+1$, a very general degree-$d$ hypersurface in $\mathbb{P}^{n+1}$ will contain no rational curves, and moreover, the locus of hypersurfaces that contain rational curves will have codimension at least $d-2n$.
\end{lemma}
\begin{proof}
See \cite[Theorem 1.2]{MR3723169-Riedl-Woolf}. The result holds over fields of arbitrary characteristics.
\end{proof}

\begin{lemma}\label{Voisin-lemma}
Over fields of characteristic zero, for $n\geq3$ and $d\geq 2n$, a very general degree-$d$ hypersurface in $\mathbb{P}^{n+1}$ will contain no rational curves.
\end{lemma}
\begin{proof}
See \cite[Theorem 0.4]{MR1420353-Voisin} and \cite[Theorem 1]{MR1669712-Voisin-Correction}. See \cite[Section 4]{MR3723169-Riedl-Woolf} for some issues on extending this result to positive characteristics.
\end{proof}

\begin{lemma}\label{Gabber-Liu}Let $X$ be an algebraic variety that geometrically contains no rational curves. Let $Y$ be a regular algebraic variety, then every rational map $f\colon Y\dashrightarrow X$ extends to a morphism $f\colon Y\to X$.
\end{lemma}
\begin{proof}
Apply \cite[Proposition 6.2]{MR3347315-Gabber-Liu} to the projection $X\times Y\to Y$.
\end{proof}

\begin{lemma}\label{Beauville}
Let $k$ be a field of characteristic zero. Let $X\subset\mathbb{P}^{n+1}$ be a smooth hypersurface of degree $d$. Let $f\colon X\to X$ be a self-morphism. If $n\geq1$ and $d\geq2$, then either $f$ is birational or $\dim f(X)< n$.
\end{lemma}
\begin{proof}For $k=\mathbb{C}$, see \cite[Theorem]{MR1809497-Beauville}. For arbitrary fields in characteristic $0$, we apply Lefschetz's principle.
\end{proof}

\begin{lemma}\label{Pic-free}
Let $k$ be a field of characteristic zero. Let $X$ be a smooth projective variety such that $\mathrm{Pic}(X)\cong\mathbb{Z}$, then every birational morphism $f\colon X\to X$ is an isomorphism.
\end{lemma}
\begin{proof}
If $f$ is not an isomorphism, then by \cite[Lemma 2.62]{Kollar-mori}, there exists an effective divisor $F$ on $X$ such that $-F$ has positive intersection with every curve contracted by $f$, this contradicts with the assumption $\mathrm{Pic}(X)=\mathbb{Z}$.\end{proof}

\begin{lemma}\label{rigidity}
Let $Y$ be an irreducible variety and $g\colon Y\to X$ be a proper and surjective morphism. Assume that every fiber of $g$ is connected and of dimension $n$. Let $f\colon Y\to X$ be a morphism such that $f(g^{-1}(z_0))$ is a point for some $z_0\in Z$. Then $f(g^{-1}(z))$ is a point for every $z\in Z$.
\end{lemma}
\begin{proof}See \cite[Lemma 1.6]{MR1658959-Kollar-Mori}. This is known as the Rigidity Lemma.
\end{proof}

\begin{lemma}\label{auto}
Let $X$ be a field of characteristic zero. Let $X$ be a general hypersurface of degree $d$ in $\mathbb{P}^{n+1}_k$. If $n\geq2$ and $d\geq4$, then every automorphism of $X$ is trivial.
\end{lemma}
\begin{proof}See \cite{MR168559-Matsumura-Monsky}. For extension of the result to positive characteristics, see \cite{chen-2017-automorphism}.
\end{proof}
\subsection{}Here is the proof of our main result.
\begin{proof}[Proof of Theorem \ref{theorem-1}] 
The case $m=0$ has been dealt with in Proposition \ref{m=0}. In the following argument, let us assume $m\geq1$. Let $(X_0)^m$ be the $m$-th self product of the $K_0$-variety $X_0$. Let $\Pi\colon (X_0)^{m+1}\to (X_0)^m$ be the projection onto the product of the first $m$ factors. The $K_m$-variety $X_m$ can be identified with the generic fiber of $\Pi$, and $K_m$-rational points of $X_m$ can be identified with graphs of $K_0$-rational maps $f\colon(X_0)^m\dashrightarrow X_0$.

By Lemma \ref{Riedl-Woolf-lemma} and Lemma \ref{Voisin-lemma}, we know that $X_0$ does not contain rational curves geometrically. Therefore by Lemma \ref{Gabber-Liu}, every rational map $f\colon (X_0)^m\dashrightarrow X_0$ extends to a $K_0$-morphism $f\colon (X_0)^m\to X_0$.
We will show that every morphism $f\colon (X_0)^m\to X_0$ is necessarily a projection onto one of the factors, this will finish the proof of Theorem \ref{theorem-1}.

Let $L$ be an algebraic closure of $K_0$, let $X_L$ be the base change of $X_0$ to $L$, and let $f_L\colon (X_L)^m\to X_L$ be the base change of $f$ to $L$. For any fixed $L$-point $p$ on $(X_L)^m$, we consider the restriction of $f_L$ to the closed subscheme $Z=\{p\}\times X_L\subset (X_L)^{m+1}$. By Lemma \ref{Beauville}, we know that $f_L|_Z\colon Z\cong X_L\to X_L$ is either birational or non-dominant. For $n\geq2$, by Noether-Lefschetz theorem for surfaces and Lefschetz hyperplane theorem for Picard groups, we know $\mathrm{Pic}(X_L)\cong\mathbb{Z}\cdot h$, where $h=\mathcal{O}_{\mathbb{P}^{n+1}}(1)|_{X_L}$.
\begin{enumerate}[leftmargin=*]
\item If $f_L|_Z$ is birational for every choice of $p$, we show that $f$ is the projection onto the last factor $$\Pi'\colon (X_L)^{m+1}\to X_L, (p,q)\mapsto q.$$
Note that $f_L|_Z$ is an isomorphism by Lemma \ref{Pic-free}. Lemma \ref{auto} forces $f_L|_Z$ to be of the form $\{p\}\times X_L\to X_L\colon (p,q)\to q$, and therefore for any $p\in (X_L)^m(L)$ and $q\in X_L(L)$, we have $f(p,q)=q$. We conclude that $f_L=\Pi'$ by applying Lemma \ref{rigidity} to $g=\Pi'$.
\item If for some choice $p\in (X_0)^m(L)$, the restriction $f_L|_Z\colon Z\to X_L$ is not dominant, then the image $f_L(Z)\subset X_L$ is an irreducible closed subvariety of dimension $\mathrm{dim}(f_L(Z))<\mathrm{dim}(X_L)$. 

If $0<\mathrm{dim}(f_L(Z))<\mathrm{dim}(X)$, then $n=\mathrm{dim}(X)\geq2$ and $\mathrm{Pic}(X)=\mathbb{Z}\cdot h$. Since $f_L(Z)$ intersects properly with hyperplane sections, we have $f^*h=\alpha\cdot h$ for some $\alpha>0$. Let $C\subset Z$ be a curve that is contracted to a point by $f_L|_Z$. By projection formula we have $0=f_*C\cdot h=C\cdot f^*h=\alpha\cdot(C\cdot h)$, this contradicts with the ampleness of $h$.  

Now $\mathrm{dim}(f_L(Z))=0$, then $Z=\{p\}\times X_L$ is mapped to a point. Apply Lemma \ref{rigidity} to $g=\Pi$, we see that $f$ factors as $f'\circ\Pi$ for some $f'\colon (X_L)^{m-1}\to X_L$. Then we conclude from induction on $m-1$, where the base case $m-1=0$ is proved in Proposition \ref{m=0}.  
\end{enumerate}
\end{proof}

\section{Proof of Theorem \ref{theorem-2}}

\subsection{}The key input into the proof is a dimension estimate of the union of rational curves on hypersurfaces \cite{MR3723169-Riedl-Woolf}. 
Let $d,n$ be integers. Let us follow our notation in Section \ref{setup-n-marked}, and keep track of the degree and dimension of hypersurfaces by subscripts, e.g., for families of degree-$d$ hypersurfaces in $\mathbb{P}^{n+1}$, we denote $P_i$ by $P_{i,(n,d)}$.

\begin{lemma}[Riedl-Woolf]\label{lemma-Erid-Riedl-dim}
Let $R_{(n,d)}\subset P_{1,(n,d)}$ be the union of rational curves contained in fibers of $p_1\colon P_{1,(n,d)}\to P_{0,(n,d)}$, then for $n\geq1, d\geq 2n+3$ or $(n,d)=(1,4)$, we have $$\dim P_{0,(n,d)}-\dim R_{(n,d)}\geq2.$$
\end{lemma}
\begin{proof}
For $d\geq 2n+3$, a general hypersurface of degree at least $2n+3$ in $\mathbb{P}^{2n+2}$ is not Fano, therefore by \cite[Theorem 3.9]{MR3723169-Riedl-Woolf}, for $d \geq 2n+3$, the union of rational curves $R_{(2n+1,d)}$ is a proper subset of $P_{1,(2n+1,d)}$. 
By \cite[Theorem 4.1]{MR3723169-Riedl-Woolf}, we know that $R_{(2n+1-c,d)}$ has codimension at least $c+1$ in $P_{1,(2n+1-c,d)}$ for $d\geq 2n+3$ and $c\geq 0$. 
Let $c=n+1$, we see that $\dim P_{1,(n,d)}-\dim R_{(n,d)}\geq n+2$ hence $\dim P_{0,(n,d)}-\dim R_{(n,d)}\geq 2$.

For $(n,d)=(1,4)$, the loci of plane quartic curves that contains a rational component has codimension at least $2$, because plane quartic curves with at worst one node contains no rational curves, and the complement of such loci has codimension $2$. 
\end{proof}

\subsection{}\label{rat-map}Let $f\colon X_0\dashrightarrow X_0$ be a $K_0$-rational map. Let $\Gamma_f\subset X_0\times_{K_0}X_0$ be the closure of the graph of $f$, and let $Z\subset P_1\times_{P_0}P_1$ be the closure of $\Gamma_f$ in the self-product of the tautological family. By definition of graph of rational maps, the projection $\mathrm{pr}_1|_Z\colon Z\subset P_1\times_{P_0}P_1\to P_1$ is birational.

\begin{proposition}If $n\geq1,d\geq2n+3$ or $(n,d)=(1,4)$ then 
$\mathrm{pr}_1|_Z\colon Z\subset P_1\times_{P_0}P_1\to P_1$ is an isomorphism.
\end{proposition}
\begin{proof}
Let $E\subset Z$ be the locus where $\pi_1|_Z$ is not an isomorphism. Note that $P_1$ is $\mathbb{Q}$-factorial, by \cite[Corollary 2.63]{MR1658959-Kollar-Mori}, it suffices to show that $\mathrm{dim}(P_1)-\mathrm{dim}(E)\geq2$. 

Let $R\subset P_1\times_{P_0}P_1$ be the union of rational curves contained in fibers of $\mathrm{pr}_1\colon P_1\times_{P_0}P_1\to P_1$. Since $P_1$ is smooth, by Ahbyankar's lemma \cite[Proposition 1.3]{MR1658959-Kollar-Mori}, we know $E\subseteq R$ and it suffice to show $\dim(P_1)-\dim (R)\geq2$. Note that $R=R_{(n,d)}\times_{P_0,p_1}P_1$
 and that $p_1$ is a submersion, we have $\dim P_{1}-\dim (R)=\dim P_{0,(n,d)}-\dim R_{(n,d)}$. We conclude from Lemma \ref{lemma-Erid-Riedl-dim}.
\end{proof}

\subsection{}Now we are ready for the proof of Theorem \ref{theorem-2}.
\begin{proof}[Proof of Theorem \ref{theorem-2}] 
As discussed in Section \ref{rat-map}, every $K_0$-rational self-map extends to a regular section $s$ to $\mathrm{pr}_1\colon P_1\times_{P_0}P_1\to P_1$. Our goal is to show that $s$ coincides with the diagonal morphism $d_1\colon P_1\to P_1\times_{P_0}P_1$.

For every $k$-point $v\in\mathbb{P}^{n+1}(k)$, let $I_v\subset\mathcal{O}_{\mathbb{P}^{n+1}}$ be the ideal sheaf of $v$, and let
$P_v:=|\mathcal{O}_{\mathbb{P}^{k}}(d)\otimes I_v|$ be the parameter space of degree-$d$ hypersurfaces in $\mathbb{P}^{n+1}$ that passes through $v$. As $v$ runs through $\mathbb{P}^{n+1}(k)$, the subspaces $P_v$ covers $P_0$. It suffices to show that for every $v\in \mathbb{P}^{n+1}(k)$, the section $s$ coincides with $d_1$ after base change to $P_v\subset P_0$.

Following notations in Section \ref{setup-n-marked}, let us consider the canonical morphism to the ambient projective space $q_2\colon P_1\times_{P_0}P_1\to \mathbb{P}^{n+1}$. The fiber $q_2^{-1}(v)$ over $v\in\mathbb{P}^{n+1}$ consists of triples $(x,H,v)$ where $H$ represents a hypersurface that contains $v$ and $x$ is a point in $H$. We identify $q_2^{-1}(v)$ with $P_v\times_{P_0}P_1$ via $((v,H,x): x\in H, v\in H) \mapsto ((H,x): x\in H, H\in P_v)$, then the diagonal section $d_i$ transforms to the constant section $P_v\times\{v\}$, and we are finally reduced to show that the $P_v\times\{v\}$ is the only regular section to $p_2|_{P_v}\colon P_v\times_{P_0}P_1\to P_v$.

Note that a regular section $s'$ to $p_2|_{P_v}$ defines a morphism $q_1\circ s'\colon P_v\to P_1\to\mathbb{P}^{n+1}$. Since $P_v$ is a projective space of dimension strictly larger than $n+1$, we know that $q_1\circ s'$ has to be the constant morphism, and its image lies in  every hypersurface that passes $v$. This forces $s'$ to be the constant section $P_v\times \{v\}$ and finishes the proof.
\end{proof}

\section{Further discussion}
\subsection{}
We note that by Proposition \ref{m=0}, for $n\geq1$, it is not possible to choose a rationally or algebraically varying manner an unordered $m$-tuple of distinct points on every smooth, degree $d$ hypersurface in $\mathbb{P}^{n+1}$ if $d$ does not divide $m$. This partly answers \cite[Question 1.3]{farb2023rigidity}.

\subsection{}
Smooth cubic hypersurface of dimension at least $2$ over a field $K$ with $K$-rational points are $K$-unirational \cite{Kollar-unirationality-cubic}, hence have a lot of $K$-rational points. Thus the optimal expectation for Theorem \ref{theorem-1} is the following:
\begin{conjecture}\label{conj}Let $k$ be a field. Let $m\geq0$ be an integer. Let $X_m/K_m$ be the generic $m$-marked degree-$d$ hypersurface in $\mathbb{P}^{n+1}_k$. If $d\geq4$, then the set of $K_m$-rational points of $X_m$ consists of the $m$ tautological sections: $$X_m(K_m)=\{\delta_1,\cdots,\delta_m\}.$$
\end{conjecture}
We summarize our current results in a table, where checkmark $\checkmark$ means Conjecture \ref{conj} is verified in characteristic zero, and checkmark with supscript $\checkmark^+$ means Conjecture \ref{conj} is also verified in characteristic $p$ with the number of markings $m=1$.
\begin{table}[!h]
\begin{center}
\begin{tabular}{c|lllllllllll}
\hline
 $(n,d)$& 4 & 5 & 6 & 7 & 8 & 9 & 10& 11&12&13&$\cdots$\\ \hline
$1$ & $\checkmark^+$ & $\checkmark^+$ & $\checkmark^+$ & $\checkmark^+$ & $\checkmark^+$ & $\checkmark^+$ & $\checkmark^+$ & $\checkmark^+$ & $\checkmark^+$& $\checkmark^+$&$\cdots$\\
$2$  &   & $\checkmark$ & $\checkmark$ & $\checkmark^+$ & $\checkmark^+$ & $\checkmark^+$ & $\checkmark^+$&$\checkmark^+$&$\checkmark^+$&$\checkmark^+$&$\cdots$\\
$3$  &  &  & $\checkmark$ & $\checkmark$ & $\checkmark$ &$\checkmark^+$& $\checkmark^+$&$\checkmark^+$&$\checkmark^+$&$\checkmark^+$&$\cdots$\\
$4$   &  &  &  &  & $\checkmark$ & $\checkmark$ & $\checkmark$ &$\checkmark^+$&$\checkmark^+$&$\checkmark^+$&$\cdots$\\
$5$   &  &  &  &  &  &  & $\checkmark$ & $\checkmark$ 
&$\checkmark$&$\checkmark^+$&$\cdots$\\ 
$\cdots$   &  &  &  &  &  &  & & &$\cdots$ &$\cdots$&$\cdots$\\
\hline
\end{tabular}
\vspace{2mm}
\caption{Verification of Conjecture \ref{conj} for generic hypersurfaces}
\label{table1}
\end{center}
\end{table}

The case $(n,d)=(3,4)$ is interesting, as it provides an example to 
\cite[Question 1.3]{Kollar-unirationality-cubic}. We remark that the case $(n,d)=(2,4)$ can be verified as a corollary of \cite[Conjecture 1]{MR2597304-Dedieu}.

\begin{proposition} Let $k$ a be field of characteristic zero. Let $X/K$ be the generic quartic surface in $\mathbb{P}^3_k$. Let $\overline{K}$ be an algebraic closure of $K$. If every dominant $\overline{K}$-rational self-map of $X_{\overline{K}}$ is birational, then the identity map is the only $K$-rational self-map of $X$.
\end{proposition}
\begin{proof}
Let $f\colon X\dashrightarrow X$ be a $K$-rational map. If $\mathrm{dim}f(X)=0$, then $f(X)$ is a $K$-point on $X$, this would contradict with Proposition \ref{m=0}. If $\mathrm{dim}f(X)=1$, let $C=\overline{f(X)}^\nu$ be the normalization of the closure of $f(X)$. Let $\pi\colon X'\to X$ be a composition of blow-ups so that
$f'=f\circ \pi\colon X'\to C$ is a resolution of indeterminancy for $f\colon X\dashrightarrow C$. Then $H^0(C,\Omega_C^1)$ injects into $H^0(X',\Omega_{X'}^1)=H^0(X,\Omega_X^1)=0$, this forces $C$ to be a smooth curve of geometric genus $0$. The canonical divisor $K_C$ provides a zero cycle of degree $-2$ on $X$, this again contradicts with Proposition \ref{m=0}. If $\mathrm{dim}f(X)=2$, by our assumption, $f$ is birational. By \cite[V.19]{MR1406314-Beauville-book}, we know $f$ is an isomorphism and furthermore the identity morphism by Lemma \ref{auto}. 
\end{proof}

\textbf{Acknowledgements.} The work is partly done when the author was a postdoc at Leibniz University Hannover. The author is grateful to his colleagues for enlightening discussions. The author thanks the referee for pointing out that Proposition \ref{m=0} partly answers \cite[Question 1.3]{farb2023rigidity}.

\bibliographystyle{alpha}
\bibliography{references} 
\end{document}